\documentclass[]{interact}

\usepackage{epstopdf}
\usepackage[caption=false]{subfig}
\usepackage{tikz-cd} 

\newcommand{\ts}{\textsuperscript}
\def \a {{\mathfrak{a}}}
\def \g {{\mathfrak{g}}}

\usepackage[numbers,sort&compress]{natbib}
\bibpunct[, ]{[}{]}{,}{n}{,}{,}
\renewcommand\bibfont{\fontsize{10}{12}\selectfont}
\makeatletter
\def\NAT@def@citea{\def\@citea{\NAT@separator}}
\makeatother

\theoremstyle{plain}
\newtheorem{theorem}{Theorem}[section]
\newtheorem{lemma}[theorem]{Lemma}
\newtheorem{corollary}[theorem]{Corollary}

\theoremstyle{definition}

\newtheorem{example}[theorem]{Example}

\theoremstyle{remark}
\newtheorem{remark}{Remark}

\newtheorem{question}{Question}

\begin{document}


\title{Exact sequences in the cohomology of a Lie superalgebra extension}

\author{
\name{Samir Kumar Hazra and Amber Habib\thanks{CONTACT Amber Habib. Email: amber.habib@snu.edu.in}}
\affil{ Department of Mathematics, School of Natural Sciences, Shiv Nadar University, Uttar Pradesh 201314, India}
}

\maketitle

\begin{abstract}
Let $  0\rightarrow \mathfrak{a} \xrightarrow{} \mathfrak{e} \xrightarrow{} \mathfrak{g} \to 0$ be an abelian extension of the Lie superalgebra $\mathfrak{g}$. In this article we consider the problems of extending  endomorphisms of $\mathfrak{a}$ and lifting endomorphisms of $\mathfrak{g}$ to certain endomorphisms of $\mathfrak{e}$. We connect these problems to the cohomology of $\mathfrak{g}$ with coefficients in $\mathfrak{a}$ through construction of two exact sequences, which is our main result, involving various endomorphism groups and the second cohomology. The first exact sequence is obtained using the Hochschild-Serre spectral sequence corresponding to the above extension while to prove the second we rather take a direct approach. As an application of our results we obtain a description of certain automorphism groups of semidirect product Lie superalgebras.
\end{abstract}


 \begin{amscode}
 17B40; 17B56; 18G40
 \end{amscode}

\section{Introduction}

Let 
\begin{equation}
\label{extension}
    0\rightarrow \mathfrak{a} \xrightarrow{i} \mathfrak{e} \xrightarrow{p} \mathfrak{g} \to 0
\end{equation}
 be an abelian extension of a Lie superalgebra $\mathfrak{g}$ by $\mathfrak{a}$, so $\mathfrak{a}$ is abelian. Corresponding to this extension, some  problems of extension and lifting of automorphisms have been considered by the authors in \cite{HH}. In particular, under what conditions can an automorphism of $\mathfrak{a}$ be extended to certain automorphism of $\mathfrak{e}$ or when can an automorphism of $\mathfrak{g}$ be lifted to certain automorphism of $\mathfrak{e}$? The work in this direction resulted in finding the following two exact sequences\cite[Theorem 4.6]{HH} from which one would obtain some necessary and sufficient conditions for the above problems:

\begin{equation}
 \label{seq1}
     1\rightarrow Aut^{\mathfrak{a},\mathfrak{g}}(\mathfrak{e})\xrightarrow{\iota} Aut^{\mathfrak{g}}_{\mathfrak{a}}(\mathfrak{e}) \xrightarrow{\tau_1} \mathcal{C}_1\xrightarrow{\lambda_1}H^2(\mathfrak{g},\mathfrak{a})_0 \\
 \end{equation}
\begin{equation}
\label{seq2}
     1\rightarrow Aut^{\mathfrak{a},\mathfrak{g}}(\mathfrak{e})\xrightarrow{\iota} Aut^{\mathfrak{a}}(\mathfrak{e}) \xrightarrow{\tau_2} \mathcal{C}_2\xrightarrow{\lambda_2}H^2(\mathfrak{g},\mathfrak{a})_0\\
\end{equation}
There, a more general problem was also considered. Namely, under what conditions can a pair of automorphisms $(\phi,\psi) \in Aut(\mathfrak{a}) \times Aut(\mathfrak{g})$ be inducible? To this end another exact sequence\cite[Theorem 5.2]{HH}, which resembles the Wells exact sequence for groups established in \cite{Wells}, was obtained from which some necessary and sufficient conditions could be given for the inducibility of the pair $(\phi,\psi)$. For details on these problems and the notation used above we refer to \cite{HH}. This study was motivated by similar problems in group theory considered in \cite{Passi} and some references therein. For a historical view of the Wells exact sequence for groups and various related problems we refer to the article \cite{Jill}. Some authors have also considered the Wells sequence in context of several other algebraic structures, like for Lie rings in \cite{Jamali} and \cite{Jamali2}, for Lie algebras in \cite{Singh}, for 3-Lie algebras in \cite{Xu} and for Lie coalgebras in \cite{Du}. In the last of these articles the extension problem for coderivations is also considered and a Wells-like exact sequence for coderivations is obtained along with that for Lie coalgebra automorphisms. The extension problems of derivations for the Lie algebras is considered in \cite{Barati}.

 A natural question that comes to mind at this point is whether we can extend the above results to the level of endomorphisms instead of just restricting to the case of automorphisms, in the present article we consider this line of investigation. In group theory, analogous problems were considered in \cite{Mariam} and it was claimed that there findings generalize the similar exact sequences for automorphism groups established in \cite[Theorem 1]{Passi}. But the article \cite{Mariam} lacks any description of certain maps involved, namely the map $\eta$ in \cite[Corollary 2]{Mariam} and $\bar{\eta}$ in \cite[Corollary 4]{Mariam}, which is necessary to validate their claim. Also, the proofs in that article uses a well-known bijection between the sets of endomorphisms and derivations of a group\cite[Lemma 8]{Mariam}. However we do not know of any such bijection for Lie superalgebras in general, we rather prove the required bijections separately. In this article we prove for Lie superalgebras the following two exact sequences similar to \eqref{seq1} and \eqref{seq2} and provide the explicit description of the maps involved. Our description of the maps would shed some light on the above said maps of \cite{Mariam}. For the notation used here we refer to Section \ref{section2}: 
 \begin{equation}
 \label{seq3}
     1\rightarrow End^{\mathfrak{a},\mathfrak{g}}(\mathfrak{e})\xrightarrow{i} End^{\mathfrak{g}}_{\mathfrak{a}}(\mathfrak{e}) \xrightarrow{\widetilde{res}} End_\mathfrak{g}(\mathfrak{a})\xrightarrow{d}H^2(\mathfrak{g},\mathfrak{a})_0 
     \end{equation}
     \begin{equation}
     \label{seq4}
      1\rightarrow End^{\mathfrak{a},\mathfrak{g}}(\mathfrak{e})\xrightarrow{i} End^{\mathfrak{a}}(\mathfrak{e}) \xrightarrow{\sigma} End^\mathfrak{a}(\mathfrak{g})\xrightarrow{\chi}H^2(\mathfrak{g},\mathfrak{a})_0 
     \end{equation}
Here one could just take the corresponding endomorphism groups, consider appropriate maps between them and check for exactness one by one at each stage as was done earlier in \cite{HH} for Lie superalgebras and in \cite{Passi} for groups. But here what is really worth noticing is that we obtain \eqref{seq3} as a consequence of the 5-term exact sequence corresponding to the well-known Hochschild-Serre spectral sequence for the Lie superalgebra extension \eqref{extension}, thus making \eqref{seq3} very special.

Another interesting question one would ask here is whether the above exact sequence \eqref{seq3} is a true generalization of \eqref{seq1} in the sense that \eqref{seq1} can be deduced from \eqref{seq3} in a suitable way. This is not so obvious, as the maps in \eqref{seq3} are the group homomorphisms of the abelian additive structure, they can not be just restricted to the set of automorphisms to produce group homomorphisms between the corresponding automorphism groups where the group operation is composition of maps. To overcome this we use the notion of quasiregular elements of rings and deduce \eqref{seq1} from \eqref{seq3} where the ring structures of $End^{\mathfrak{a},\mathfrak{g}}(\mathfrak{e})$ and $End_\mathfrak{g}(\mathfrak{a})$ are the usual ones but that of $End^{\mathfrak{g}}_{\mathfrak{a}}(\mathfrak{e})$, which is very different from the usual one, comes from the ring structure of $Z^1(\mathfrak{e}, \mathfrak{a})_0$ via a bijection $Z^1(\mathfrak{e}, \mathfrak{a})_0 \simeq End^{\mathfrak{g}}_{\mathfrak{a}}(\mathfrak{e})$. Thus the exact sequence \eqref{seq1} will turn out to be a consequence of the Hochschild-Serre spectral sequence, this interesting fact was not known earlier. 
To prove \eqref{seq4} we rather take a direct approach and using the monoidal structures of the involved groups we obtain \eqref{seq2} as a corollary. At the end we give a description of certain automorphism groups of semidirect product of Lie superalgebras.

The article is organized as follows. In Section \ref{section2} we present various definitions and set up some notation which are needed throughout the article. Some basic notions whose definitions are not given here are defined in \cite{HH}. Here we also prove an important lemma which will be the main ingredient for the proof of Theorem \ref{1sttheorem}. Theorem \ref{1sttheorem} is our first main theorem, this is stated and proved in Section \ref{section3} establishing the first exact sequence \eqref{seq3}. Then we provide some corollaries of the theorem of which Corollary \ref{corollary} is particularly important. This corollary is proved using the notion of quasiregular elements which is also introduced here. The another exact sequence \eqref{seq4} is proved in Theorem \ref{Theorem2}, along with various corollaries of it, in Section \ref{section 4}. The last section which is Section \ref{section 5} mainly deals with an application of our results to the semidirect product Lie superalgebras. Here we prove a couple of group isomorphisms which describe certain automorphism groups of semidirect products. The section ends with two examples. To avoid unnecessary complications we shall assume the underlying field $F$ to be of characteristic $0$.

\section{Preliminaries}
\label{section2}

We refer to \cite[Section 2]{HH} for definitions of Lie superalgebras, modules over Lie superalgebras and related notions. A linear map $\phi: \mathfrak{g}(={\mathfrak{g}_0 \oplus \mathfrak{g}_1})\to \mathfrak{h}(=\mathfrak{h}_0 \oplus \mathfrak{h}_1)$ is said to be homogeneous of degree $0$ if $\phi(\mathfrak{g}_i)\subseteq \mathfrak{h}_i $ for $i \in \{0,1\}$.  We also call such a map an even map in short. Then $\phi$ is called a Lie superalgebra homomorphism if $\phi$ is even and $\phi([x,y])=[\phi(x), \phi(y)]$ for all $x,y \in \mathfrak{g}$. As usual, for a Lie superalgebra $\mathfrak{e}$ we define $End(\mathfrak{e})$ to be the set of all Lie superalgebra homomorphisms from $\mathfrak{e}$ to itself. Also for an ideal $\mathfrak{a}\subseteq \mathfrak{e}$ define $$End_\mathfrak{a}(\mathfrak{e}):= \left\{ \phi \in End(\mathfrak{e}) : \phi(\mathfrak{a})\subseteq \mathfrak{a}  \right\}.$$ 
Now consider the extension \eqref{extension} and let $\phi \in End_\mathfrak{a}(\mathfrak{e})$. Then $\phi$ induces two maps $\phi|_\mathfrak{a}\in End(\mathfrak{a})$, the restriction of $\phi$ to $\mathfrak{a}$ and $\tilde{\phi} \in End(\mathfrak{g})$ defined by $\tilde{\phi}(x):= p\phi s(x)$ where $s$ is any section of the map $p$ i.e. $s:\mathfrak{g}\to \mathfrak{e}$ is linear satisfying $ps=1$. We shall fix such a section $s$ throughout the article and without loss of generality we shall take our section to be an even map. Now it can be checked that $\tilde{\phi}$ is well-defined, that is it does not depend on the choice of the section. Clearly the maps $\phi|_\mathfrak{a}$ and $\tilde{\phi}$ can be characterized by the following commutative diagram:

\begin{center}
\begin{tikzcd}
   0\arrow{r} & \mathfrak{a}\arrow{r}{i}\arrow[dotted]{d}{\phi|_\mathfrak{a}} & \mathfrak{e}\arrow{r}{p}\arrow{d}{\phi} & \mathfrak{g}\arrow{r}\arrow[dotted]{d}{\tilde{\phi}} & 0 \\
  0\arrow{r} & \mathfrak{a}\arrow{r}{i} & \mathfrak{e}\arrow{r}{p} & \mathfrak{g}\arrow{r} & 0 
\end{tikzcd}
\end{center}

Define $$End^{\mathfrak{a},\mathfrak{g}}(\mathfrak{e}):= \left\{\phi \in End_\mathfrak{a}(\mathfrak{e}): \phi|_\mathfrak{a}=id \text{ and } \tilde{\phi}=id \right\}$$ then $5$-lemma implies $End^{\mathfrak{a},\mathfrak{g}}(\mathfrak{e}) = Aut^{\mathfrak{a},\mathfrak{g}}(\mathfrak{e})$. Also define $$End_\mathfrak{a}^{\mathfrak{g}}(\mathfrak{e}):= \left\{\phi \in End_\mathfrak{a}(\mathfrak{e}): \tilde{\phi}=id \right\}$$ and $$End^\mathfrak{a}(\mathfrak{e}):= \left\{\phi \in End_\mathfrak{a}(\mathfrak{e}):{\phi}|_\mathfrak{a}=id \right\}$$

Now for an abelian extension \eqref{extension} there is an induced well-defined action of $\mathfrak{g}$ on $\mathfrak{a}$ given by $g\cdot a:= [s(g), a]$ where $s$ is a section of $p$, $g\in \mathfrak{g}$ and $a\in \mathfrak{a}$. Therefore we can talk about the cohomology groups $H^n(\mathfrak{g}, \mathfrak{a})$. See \cite[Section 2.1]{HH} for details on constructions, descriptions and gradings of low dimensional cohomologies. We also set $$End^\mathfrak{a}(\mathfrak{g}):= \left\{\psi \in End(\mathfrak{g}): \psi(g) \cdot a=g\cdot a \, \forall g\in \mathfrak{g}, a \in \mathfrak{a} \right\}$$ which is the set of all endomorphisms of $\mathfrak{g}$ which does not alter the action of $\mathfrak{g}$ on $\mathfrak{a}$. 

Let us again consider the extension \eqref{extension} and let $M$  be an $\mathfrak{e}$-module. Then it is well-known that there is a corresponding Hochschild-Serre spectral sequence whose $(p,q)$-term on the $2\ts{nd}$-page is given by $E^{p,q}_2=H^p(\mathfrak{g}, H^q(\mathfrak{a}, M))$ converging to $H^{p+q}(\mathfrak{e}, M)$, see \cite[Section 16.6.]{Musson}. In our case the module is $\mathfrak{a}$ with the adjoint action of $\mathfrak{e}$ on it. Clearly with this action $\mathfrak{a}$ is $\mathfrak{a}$-invariant, i.e, $\mathfrak{a}^\mathfrak{a}=\mathfrak{a}$. Here for a $\mathfrak{g}$-module $M$, $M^\mathfrak{g}$ denotes the set of invariants of the action defined by $M^\mathfrak{g}:=\{m\in M | g\cdot m=0\, \forall g\in \mathfrak{g}\}$. Then the $5$-term exact sequence corresponding to the above spectral sequence becomes 
\begin{equation}
\label{5term}
    0\to H^1(\mathfrak{g}, \mathfrak{a}) \xrightarrow{inf} H^1(\mathfrak{e}, \mathfrak{a}) \xrightarrow{res} {H^1(\mathfrak{a}, \mathfrak{a})}^\mathfrak{g}\xrightarrow{d} H^2(\mathfrak{g}, \mathfrak{a}) \xrightarrow{inf} H^2(\mathfrak{e}, \mathfrak{a})
\end{equation}
where $inf$ and $res$ are the inflation and restriction maps respectively and $d$ is induced by the spectral sequence. Since all the maps involved here respect the grading on cohomology groups we immediately get 
\begin{equation}
\label{even5term}
    0\to H^1(\mathfrak{g}, \mathfrak{a})_0 \xrightarrow{inf} H^1(\mathfrak{e}, \mathfrak{a})_0 \xrightarrow{res} {H^1(\mathfrak{a}, \mathfrak{a})}_0^\mathfrak{g}\xrightarrow{d} H^2(\mathfrak{g}, \mathfrak{a})_0 \xrightarrow{inf} H^2(\mathfrak{e}, \mathfrak{a})_0
\end{equation}
Now let $\mathfrak{g}$ be a Lie superalgebra and $M=M_0\oplus M_1$ be a module over $\mathfrak{g}$. An even linear map $f:\mathfrak{g}\to M$ is said to be a derivation if 
\begin{equation}
\label{derivation}
    f([x,y])= x\cdot f(y) - (-1)^{|x||y|} y\cdot f(x) \text{ for all homogeneous } x,y \in \mathfrak{g}
\end{equation}
where $|\cdot|$ denotes the homogeneous degree. These are also the even $1$-cocycles whose collection is denoted by $Z^1(\mathfrak{g},M)_0$. If in particular $f$ is given by $f(x)=x\cdot m$ for some $m\in M_0$ we call $f$ an inner derivation, clearly such an $f$ is even. Let us denote of the set all the derivations by $Der(\mathfrak{g},M)$ and the inner derivations by $Inn(\mathfrak{g},M)$, so $Der(\mathfrak{g},M)=Z^1(\mathfrak{g},M)_0$ . Then it is well-known that $H^1(\mathfrak{g}, M)_0=\frac{Der(\mathfrak{g},M)}{Inn(\mathfrak{g},M)}$.

Now we prove the following lemma which will be used in Section \ref{section3}:

\begin{lemma}
\label{lemma reduced5term}
 \begin{equation}
 \label{lemma cocycle}
    0\to Z^1(\mathfrak{g}, \mathfrak{a})_0 \xrightarrow{inf} Z^1(\mathfrak{e}, \mathfrak{a})_0 \xrightarrow{res} End_\mathfrak{g}(\mathfrak{a})\xrightarrow{d} H^2(\mathfrak{g}, \mathfrak{a})_0 \xrightarrow{inf} H^2(\mathfrak{e}, \mathfrak{a})_0
\end{equation}
 is an exact sequence of abelian groups where $End_\mathfrak{g}(\mathfrak{a})$ denotes the set of $\mathfrak{g}$-module endomorphisms  of $\mathfrak{a}$.
\end{lemma}

\begin{proof}
As $\mathfrak{a}$ is abelian the adjoint action of $\mathfrak{a}$ on $\mathfrak{a}$ is trivial. So the only inner derivation is the trivial map and any even linear map $f:\mathfrak{a} \to \mathfrak{a}$ satisfies \eqref{derivation}. Therefore $H^1(\mathfrak{a},\mathfrak{a})_0^\mathfrak{g}={Z^1(\mathfrak{a},\mathfrak{a})}^\mathfrak{g}_0={End(\mathfrak{a})}^\mathfrak{g}$. Now, as ${End(\mathfrak{a})}^\mathfrak{g}=End_\mathfrak{g}(\mathfrak{a})$ we have $H^1(\mathfrak{a},\mathfrak{a})_0^\mathfrak{g}=End_\mathfrak{g}(\mathfrak{a})$. Also let $\phi \in  Z^1(\mathfrak{e}, \mathfrak{a})_0$. Then clearly $res(\phi)=\phi|_\mathfrak{a}=0$ if and only if $\phi$ can be quotiented out to a map in $Z^1(\mathfrak{g}, \mathfrak{a})_0$. It only remains to show the exactness at the third term. For that take two derivations $\phi,\psi \in Z^1(\mathfrak{e}, \mathfrak{a})_0$ which are cohomologous in $H^1(\mathfrak{e}, \mathfrak{a})_0$. Therefore $\phi-\psi=h$ where $h\in Inn(\mathfrak{e},\mathfrak{a})$ is an inner derivation. But then $res(\phi)=res(\psi)$ as $res(h)=h|_\mathfrak{a}=0$, $\mathfrak{a}$ being abelian. This proves $res(H^1(\mathfrak{e}, \mathfrak{a})_0)=res \left(Z^1(\mathfrak{e}, \mathfrak{a})_0\right)$. Consequently \eqref{lemma cocycle} is exact.
\end{proof}

\section{Construction of the first exact sequence} 
\label{section3}
In this section we are going to prove our first main result establishing an exact sequence which we have already seen in \eqref{seq3}.

\begin{theorem}
\label{1sttheorem}
The following is an exact sequence of (additive) abelian groups: 
\begin{equation}
\label{Theorem1}
     1\to End^{ \mathfrak{a},\mathfrak{g}}(\mathfrak{e}) \xrightarrow{i} End^{\mathfrak{g}}_\mathfrak{a}(\mathfrak{e})\xrightarrow{\widetilde{res}} End_\mathfrak{g}(\mathfrak{a})\xrightarrow{d} H^2(\mathfrak{g}, \mathfrak{a})_0.
\end{equation}
Moreover the first three terms have certain multiplicative structures with respect to which they are rings and all the maps involved are ring homomorphisms except $d$ which is just a group map.
\end{theorem}
The usefulness of imposing the extra multiplicative structure will very soon be clear to us. But before we proceed to that deeper fact, first we note that though the ring structure of $End_\mathfrak{g}(\mathfrak{a})$ is the obvious one induced from the usual ring structures of $End(\mathfrak{a})$,  that of $End^{\mathfrak{g}}_\mathfrak{a}(\mathfrak{e})$(and hence of $End^{\mathfrak{a}, \mathfrak{g}}(\mathfrak{e})$) is going to be very different from the usual one. We will introduce this structure now. We start with the following lemma:
\begin{lemma}
The following is a set bijection:
\begin{equation}
\label{bijection}
   Z^1(\mathfrak{e}, \mathfrak{a})_0 \simeq  End^{\mathfrak{g}}_\mathfrak{a}(\mathfrak{e}).
\end{equation}

\end{lemma}
\begin{proof}
Considering the map $\Psi: Z^1(\mathfrak{e}, \mathfrak{a})_0 \to End^{\mathfrak{g}}_\mathfrak{a}(\mathfrak{e}) $ defined by $\Psi(h)(x):=h(x) + x$ one can easily establish the required bijection. 
\end{proof}
Now with the obvious additive abelian group structure, $Z^1(\mathfrak{e}, \mathfrak{a})_0$ becomes a ring where the product of two elements is given by the composition of maps, here we identify $\mathfrak{a}$ with $i(\mathfrak{a})$. The ring axioms are very easy to check but what needs to be proved is that the composition is again a derivation. This is proved in the following lemma:
\begin{lemma}
For $f,g \in  Z^1(\mathfrak{e}, \mathfrak{a})_0 $, $f\circ g \in Z^1(\mathfrak{e}, \mathfrak{a})_0 $.
\end{lemma}
\begin{proof} Clearly $f\circ g$ is an even map.
Now for homogeneous $x,y \in \mathfrak{e}$, 
\begin{align*}
 & f\circ g([x,y]) \\
 &=f([x,g(y)] - (-1)^{|x||y|} [y, g(x)]  )   \\
 &= [x, f\circ g(y)]- (-1)^{|x||y|}[g(y), f(x)]- (-1)^{|x||y|}([y, f \circ g(x)]- (-1)^{|x||y|}[g(x), f(y)])  \\
 &=[x, f\circ g(y)]- (-1)^{|x||y|}[y, f \circ g(x)] \text{ as $\mathfrak{a}$ is abelian}
\end{align*}
This proves the lemma.
\end{proof}

Next we push this ring structure of $Z^1(\mathfrak{e}, \mathfrak{a})_0$ via the above bijection to get a ring structure on $End^{\mathfrak{g}}_\mathfrak{a}(\mathfrak{e})$, the explicit descriptions of the operations are given below: 

Let $f,g \in End^{\mathfrak{g}}_\mathfrak{a}(\mathfrak{e})$. Then applying the bijection \eqref{bijection} the addition, denoted by $\boxplus$, of these two elements must be given by
\begin{equation}
\label{boxplus}
    {f \boxplus g}(x):=\Psi(\Psi^{-1}(f)+\Psi^{-1}(g))(x)
    = \Psi^{-1}(f)(x)+ \Psi^{-1}(g)(x)+x
    = f(x)-x+g(x).
\end{equation}
Notice that the identity map is the zero element in  this ring.

Also the product, denoted by $\boxtimes$, must be given by
\begin{align}
\label{boxtimes}
     {f \boxtimes g}(x):&=\Psi(\Psi^{-1}(f) \circ \Psi^{-1}(g))(x) \notag \\
     &= (f-id)\circ (g-id)(x)+x \notag \\
     &= f\circ g (x) -f(x)-g(x) +2x.
    \end{align}
Now as we replace $Z^1(\mathfrak{e}, \mathfrak{a})_0$ in \eqref{lemma cocycle} by $End^{\mathfrak{g}}_\mathfrak{a}(\mathfrak{e})$(with the above ring structure), the map $res$ gets changed accordingly. This changed map we denote by $\widetilde{res}$, clearly 
\begin{equation}
\label{description of restilde}
    \widetilde{res}(f)(x)= f(x)-x\,\,\,x\in \mathfrak{a}.
\end{equation}
So it is not just the restriction of $f$ to $\mathfrak{a}$. The next lemma shows that $\widetilde{res}$ is a ring homomorphism.
\begin{lemma}
\label{hom of rings}
The map $\widetilde{res}: End^{\mathfrak{g}}_\mathfrak{a}(\mathfrak{e}) \to End_\mathfrak{g}(\mathfrak{a}) $ is a homomorphism of rings.
\end{lemma}
\begin{proof}
Using the bijection \eqref{bijection}, the lemma follows from the fact that\\ $res:Z^1(\mathfrak{e}, \mathfrak{a})_0 \to End_\mathfrak{g}(\mathfrak{a})$ is a ring homomorphism as $res(f\circ g)=f \circ g |_\mathfrak{a}=f|_\mathfrak{a}\, \circ\, g|_\mathfrak{a}=res(f)\circ \, res(g)$.
\end{proof}
The next lemma shows that \eqref{Theorem1} is exact at the $2\ts{nd}$ term.
\begin{lemma}
\label{ker of tilderes}
$ker (\widetilde{res})= End^{ \mathfrak{a},\mathfrak{g}}(\mathfrak{e}). $
\end{lemma}
\begin{proof}
This follows from exactness of \eqref{lemma cocycle}, \cite[Lemma 5.1]{HH} and the commutativity of the diagram 
\begin{tikzcd}
  Z^1(\mathfrak{g}, \mathfrak{a})_0 \arrow{r}{inf} \arrow{d}{\Psi} &  Z^1(\mathfrak{e}, \mathfrak{a})_0 \arrow{d}{\Psi}  \\
  End^{ \mathfrak{a},\mathfrak{g}}(\mathfrak{e}) \arrow{r}{i} & End^{\mathfrak{g}}_\mathfrak{a}(\mathfrak{e})
\end{tikzcd}

Here we used the observation $End^{ \mathfrak{a},\mathfrak{g}}(\mathfrak{e})=Aut^{ \mathfrak{a},\mathfrak{g}}(\mathfrak{e})$ implied by 5-lemma.
\end{proof}

\subsection{Proof of Theorem \ref{1sttheorem}}
\begin{proof}
From the construction of the map $\widetilde{res}$ clearly $img(res)= img(\widetilde{res})$. Therefore the exactness of \eqref{Theorem1} at the $3\ts{rd}$ term follows from the corresponding exactness of \eqref{lemma cocycle}. Exactness at the $2\ts{nd}$ term follows from Lemma \ref{ker of tilderes}. Finally from Lemma \ref{hom of rings} we have the first two maps are homomorphism of rings, here the ring structure of $End^{\mathfrak{g}}_\mathfrak{a}(\mathfrak{e})$ is given by \eqref{boxplus} and \eqref{boxtimes}. This proves the theorem.
\end{proof}

We have the following corollary of Theorem \ref{1sttheorem}:
\begin{corollary} The following is an exact sequence of groups:
\label{corollary}
\begin{equation}
\label{cor1}
    1\rightarrow Aut^{\mathfrak{a},\mathfrak{g}}(\mathfrak{e})\xrightarrow{i} Aut^{\mathfrak{g}}_{\mathfrak{a}}(\mathfrak{e}) \xrightarrow{\widetilde{res}} Aut_\mathfrak{g}(\mathfrak{a})\xrightarrow{d} H^2(\mathfrak{g},\mathfrak{a})_0
\end{equation}
  \end{corollary}

We claim that this is exactly the first exact sequence given in  \cite[Theorem 4.6]{HH}, see \eqref{seq1}. In \cite{HH} it was proved by considering certain maps between the objects and checking exactness at each term, this is the similar way how it was done for the group case in \cite{Passi}. The most important point to notice here is that we obtain the same exact sequence \eqref{seq1} as a consequence of the Hochschild-Serre spectral sequence though the proof of this fact is not so obvious. Evidently we could not just restrict the maps in \eqref{Theorem1} to the group of units in the rings and get \eqref{cor1} as the ring structure of $End^{\mathfrak{g}}_{\mathfrak{a}}(\mathfrak{e})$ is quite different. For example, we miss the identity map from $End^{\mathfrak{g}}_{\mathfrak{a}}(\mathfrak{e})$ while considering the units, the identity map being the ``zero" element of the current ring structure. To overcome this we make use of the notion of quasiregular elements in rings.

Let $R$ be a ring. Define an operation  $\ast: R \times R \to R$ by $x \ast y =x +y+xy$ for $x,y \in R$. An element $x \in R$ is said to be quasiregular if there exists an element $y \in R$ such that $x \ast y= y \ast x=0$. Denote the set of all such elements by $QR(R)$, then $(QR(R),\ast)$ is a group. If $R$ is a ring with unity then we have an isomorphism $(QR(R), \ast)\simeq U(R)$ of groups given by the map $r\mapsto 1+r$ where $U(R)$ denotes the group of units of $R$.
See \cite{Kaplansky} for details on quasiregular elements.

The following lemma describes the quasiregular elements of $End^\mathfrak{g}_\mathfrak{a}(\mathfrak{e})$:
\begin{lemma}
\label{qr}
 $\left( QR(End^\mathfrak{g}_\mathfrak{a}(\mathfrak{e})), \ast \right) \simeq Aut^\mathfrak{g}_\mathfrak{a}(\mathfrak{e})$ as groups.
\end{lemma}
 \begin{proof}
 To prove this we show that the $\ast$ operation in $End^\mathfrak{g}_\mathfrak{a}(\mathfrak{e})$ turns out to be the composition of maps. 
 
 Let $f,g \in End^\mathfrak{g}_\mathfrak{a}(\mathfrak{e})$. Then 
 \begin{align*}
     f\ast g(x) &=f \boxplus g  \boxplus (f\boxtimes g)(x)\\
     &=(f(x)-x+g(x)) \boxplus (f\circ g(x) -f(x) -g(x)+2x) \text{ [ using \eqref{boxplus} and \eqref{boxtimes}}]\\
     &= f\circ g(x) \text{ [ using \eqref{boxplus}]}
     \end{align*}
 \end{proof}
 
 \subsection{Proof of Corollary \ref{corollary}}
 Here we give a proof of the above corollary.
 \begin{proof}
 First we restrict the map $\widetilde{res}$ of \eqref{Theorem1} to the set of quasiregular  elements. By doing so and using Lemma \ref{qr} we get a group homomorphism 
 \begin{equation}
 \widetilde{res}: Aut^\mathfrak{g}_\mathfrak{a}(\mathfrak{e}) \to (QR(End_\mathfrak{g}(\mathfrak{a})), \ast). 
 \end{equation}
 Now as $End_\mathfrak{g}(\mathfrak{a})$ is a ring with unity, the quasiregular elements in that ring are just the invertible elements and we have the isomorphism 
 \begin{equation}
 \label{isom qr and autg(a)}
     (QR(End_\mathfrak{g}(\mathfrak{a})), \ast) \simeq Aut_\mathfrak{g}(\mathfrak{a}) 
 \end{equation}
 given by $\phi \mapsto id+\phi $. 
 Thus the map 
 \begin{equation}
 \label{ker in cor1}
     \widetilde{res}: Aut^\mathfrak{g}_\mathfrak{a}(\mathfrak{e}) \to Aut_\mathfrak{g}(\mathfrak{a})
 \end{equation}
 becomes just the restriction map given by $\widetilde{res}(\phi)=\phi|_\mathfrak{a}$. Also by Lemma \ref{ker of tilderes} the kernel of the map $\widetilde{res}$ given in \eqref{ker in cor1} is $Aut^{\mathfrak{a},\mathfrak{g}}(\mathfrak{e})$. So \eqref{cor1} is exact at the 2\ts{nd} term. For the exactness at the 3\ts{rd} term let $\psi \in ker\, d$ in \eqref{cor1}. Then by the exactness of \eqref{Theorem1} $\psi= \widetilde{res}(\phi)$ for some $\phi \in End^\mathfrak{g}_\mathfrak{a}(\mathfrak{e})$. But since also $\psi \in Aut_\mathfrak{g}(\mathfrak{a})$, from the description of $\widetilde{res}$ in \eqref{description of restilde} we have $\phi|_\mathfrak{a}$ is an automorphism. Therefore by 5-lemma it follows that in particular $\phi \in Aut^\mathfrak{g}_\mathfrak{a}(\mathfrak{e})$. Clearly then $\psi= \widetilde{res}(\phi)$ for some $\phi \in Aut^\mathfrak{g}_\mathfrak{a}(\mathfrak{e})$. Consequently $\eqref{cor1} $ is exact at the 3\ts{rd} term. This completes the proof of the corollary.
 \end{proof}
 Before we proceed further, it follows from \cite{HH} that the object $\mathcal{C}_1$ in \eqref{seq1} is exactly $Aut_\mathfrak{g}(\mathfrak{a})$. Now to validate our claim that the exact sequence \eqref{cor1} is exactly same as \eqref{seq1}, we need to show that the involved maps are same. It is clear that the first map is just inclusion. The map $\tau_1$ in \eqref{seq1} is actually the usual restriction map follows from the construction given in the paragraph just before \cite[Lemma 4.5]{HH}, which is the same map as $\widetilde{res}$ given in \eqref{ker in cor1}. So what we are only left to show is that the map $d$ in \eqref{cor1} is same as $\lambda_1$.
 
 For this purpose  we need to have an exact description of the map $d$, the description of $\lambda_1$ can be found in the discussion just before \cite[Lemma 4.3]{HH}. We remember that $d$ is induced by the Hochschild-Serre spectral sequence and in the case of a general module $M$ over $\mathfrak{e}$ giving an explicit description of $d$ seems to be a difficult problem. But in our case the module is $\mathfrak{a}$ and in particular the action on $\mathfrak{a}$ as a module over itself is trivial, $\mathfrak{a}$ being abelian. In this case a cocycle description of the map $d$ can be obtained from the following theorem:

\begin{theorem}\cite[Theorem 16.6.7]{Musson}
\label{Theorem Musson}

Let $0\rightarrow \mathfrak{n} \xrightarrow{i} \mathfrak{k} \xrightarrow{p} \mathfrak{q} \to 0$ be an extension of Lie superalgebra and $M$ a $\mathfrak{k}$-module. Assume that $\mathfrak{n}$ is abelian and $\mathfrak{n}M=0$. Then the differential $d_2:E^{p,1}_2 \to E^{p+2,0}_2$ is given by $$\theta_{p+2,0}(d_2(u))=-c \cup \theta_{p,1}(u)$$
where $\cup$ is the cup product and $c \in H^2(\mathfrak{q},\mathfrak{n})$ is the cohomology class of the 2-cocycle associated to the given extension.
\end{theorem}
Here $\theta_{p,q}:E^{p,q}_2 \to H^p(\mathfrak{q}, H^q(\mathfrak{n}, M))$ is an isomorphism. For details see \cite[Chapter 16]{Musson}.

In our case the extension is \eqref{extension} and the map $d$ in \eqref{5term} is the map $d:E^{0,1}_2 \to E^{2,0}_2$ of the Hochschild-Serre spectral sequence. So we will apply the above theorem for $(p,q)\equiv(0,1)$ and obtain a cocycle description of $d$. In the following lemma we give a description of the cup product in a particular case. The objects $C^n(-,\, -)$ denotes the the group of $n$-cochains, see \cite{HH} or see \cite{Musson} for more details.
\begin{lemma}
The cup product on cochains 
\begin{equation}
    \cup: C^2(\g, \a) \times C^0(\g,  C^1(\a, \a)) \to C^2(\g, \a)
\end{equation}
can be given by 
\begin{equation}
    (h\cup f)(x,y)=(-1)^{|f|(|x|+|y|)}(-1)^{|f||h(x,y)|}f(h(x,y))
\end{equation}
for $x,y\in \g$, $h\in C^2(\g, \a)$ and $f\in  C^1(\a, \a)$.
\end{lemma}
\begin{proof}
Let $F$ be the underlying field. We can identify $C^0(\g,  C^1(\a, \a))=Hom_F(F, C^1(\a, \a))$, the space of $F$-linear maps from $F$ to $C^1(\a, \a)$, with $C^1(\a, \a)$ via the following identification: for $ h\in Hom_F(F, C^1(\a, \a))$ we identify $h$ with $h(1)\in  C^1(\a, \a)$. Clearly, if $h$ is homogeneous we have $|h|=|h(1)|$, as it is always assumed that the $\mathbb{Z}_2$-grading of $F$ is concentrated only at $0\ts{th}$ component. Then taking $k=1$ in \cite[Example 16.5.9]{Musson} we obtain the pairing $\star : \a \otimes C^1(\a, \a) \to \a $ given by 
\begin{equation}
\label{pairing}
     a \otimes h \mapsto (-1)^{|a||h|} h(a).
\end{equation}
Now we take $n=2$, $q=2$ and $p=0$. Also let $f\in C^2(\g, \a)$ and $h\in  C^1(\a, \a)$. Then again from \cite[Example 16.5.9]{Musson} we have 
\begin{equation}
\label{special case of cup product}
\begin{split}
    (f\cup h)(x,y)&= (-1)^{|h|(|x|+|y|)}f(x,y)\star h\\
    &=(-1)^{|h|(|x|+|y|)}(-1)^{|h||f(x,y)|}h(f(x,y))\,\,[\text{using\, \eqref{pairing}}]
    \end{split}
\end{equation}
for $x,y \in \g$. Here we have used the facts that $\Xi$ of \cite[Example 16.5.9]{Musson} is given by the identity permutation $\{e\}$ in this particular case and the inversion set of identity permutation $Inv(e)= \emptyset$, the empty set.
\end{proof}

We use the above formula for cup product to give a description of the map $d$.
\begin{lemma}
\label{description of d}
The map $d:H^1(\mathfrak{a}, \mathfrak{a})_0^\mathfrak{g} \to H^2(\mathfrak{g}, \mathfrak{a})_0$ is given by $d([h])= -[h\circ \beta]$ where $\beta$ is a 2-cocycle corresponding to \eqref{extension} and $[\cdot]$ denotes the cohomology class. 
\end{lemma}
Here $\beta:\mathfrak{g}\times \mathfrak{g} \to \mathfrak{a}$ is given by $\beta(x,y):=[\eta(x), \eta(y)]-\eta[x,y]\,; x,y \in \mathfrak{g}$ for any section $\eta$ of the map $p$ in \eqref{extension}. Such a $\beta\in Z^2(\mathfrak{g},\mathfrak{a})_0$, see \cite[Lemma 4.1]{HH}.
\begin{proof} We use the identifications $H^0(\mathfrak{g},H^1(\mathfrak{a},\mathfrak{a}))=H^1(\mathfrak{a},\mathfrak{a})^\mathfrak{g}$ and $H^2(\mathfrak{g},H^0(\mathfrak{a},\mathfrak{a}))= H^2(\mathfrak{g},\mathfrak{a}^\mathfrak{a})=H^2(\mathfrak{g},\mathfrak{a})$ to obtain the following commutative diagram:
\begin{center}
\begin{tikzcd}
  E_2^{0,1}\arrow{r}{d_2} \arrow{d}{\theta_{0,1}} &  E_2^{2,0} \arrow{d}{\theta_{2,0}}  \\
  H^1(\mathfrak{a},\mathfrak{a})^\mathfrak{g} \arrow{r}{d} & H^2(\mathfrak{g},\mathfrak{a})
\end{tikzcd}.
\end{center}
\noindent Let us take $u\in E^{0,1}_2$ and let $\theta_{0,1}(u)=[h]$, the cohomology class of $h$ for some $h\in C^1(\mathfrak{a},\mathfrak{a})$.
Then by Theorem \ref{Theorem Musson} we have 
\begin{align*}
    d([h])(x,y)&=-[\beta \cup h](x,y)\\
    &=-(-1)^{|h|(|x|+|y|)}(-1)^{|h||\beta(x,y)|}[h\circ \beta](x,y) \hspace{.5cm}(\text{using \, \eqref{special case of cup product}})\\
    &=-[h\circ \beta](x,y) \hspace{.5cm}(\text{as $\beta\in Z^2(\mathfrak{g},\mathfrak{a})_0$,\,$|\beta(x,y)|=|x|+|y|$}).
\end{align*}
Now since the maps involved respect the grading of each cohomology group we have the lemma.
\end{proof}

We have seen earlier in the proof of Lemma \ref{lemma reduced5term} that $H^1(\mathfrak{a},\mathfrak{a})^\mathfrak{g}_0=Z^1(\mathfrak{a},\mathfrak{a})_0^\mathfrak{g}=End_\mathfrak{g}(\mathfrak{a})$. So the above lemma gives a cocycle description of the map $d: End_\mathfrak{g}(\mathfrak{a}) \to  H^2(\mathfrak{g}, \mathfrak{a})_0 $ in \eqref{lemma cocycle}. Therefore when we restrict to quasiregular elements we have the same description of the map $d:(QR(End_\mathfrak{g}(\mathfrak{a})), \ast) \to H^2(\mathfrak{g}, \mathfrak{a})_0$. But as we have identified $(QR(End_\mathfrak{g}(\mathfrak{a})), \ast)$ and $ Aut_\mathfrak{g}(\mathfrak{a})$ in \eqref{cor1} via the isomorphism \eqref{isom qr and autg(a)}, the description of $d:Aut_\mathfrak{g}(\mathfrak{a}) \to H^2(\mathfrak{g}, \mathfrak{a})_0 $ gets changed accordingly. Clearly using the commutativity of the diagram bellow we have $d:Aut_\mathfrak{g}(\mathfrak{a}) \to H^2(\mathfrak{g}, \mathfrak{a})_0 $ given by $d(\phi)= [\beta - \phi \circ \beta]$. This is negative of the map $\lambda_1$ of \eqref{seq1} given in  \cite{HH}. This establishes  \eqref{seq1} as a consequence of our Theorem \ref{1sttheorem} upto a sign of $\lambda_1$ which anyway does not affect the exactness at all. 
\begin{center}
\begin{tikzcd}
  (QR(End_\mathfrak{g}(\mathfrak{a})), \ast) \arrow[dd, "h\, \mapsto \, id+\,h", "\simeq" ' ]  \arrow[dr, "{d(h)= -[h\circ \beta]}" ] &  \\
  & H^2(\mathfrak{g}, \mathfrak{a})_0 \\
  Aut_\mathfrak{g}(\mathfrak{a}) \arrow[ur, dashed, "d", "{d(\phi)= [\beta - \phi \circ \beta]}" '] &
\end{tikzcd}
\end{center}

\begin{subsection}{More corollaries}
Here we present several other consequences of Theorem \ref{1sttheorem}.
\begin{corollary}
\label{cor lie algebra} For an abelian extension of Lie algebra $ 0\rightarrow \mathfrak{a} \xrightarrow{} \mathfrak{e} \xrightarrow{} \mathfrak{g} \to 0$  we have the following exact sequence of rings 
\begin{equation}
    0\to End^{ \mathfrak{a},\mathfrak{g}}(\mathfrak{e}) \xrightarrow{i} End^{\mathfrak{g}}_\mathfrak{a}(\mathfrak{e})\xrightarrow{\widetilde{res}} End_\mathfrak{g}(\mathfrak{a})\xrightarrow{d} H^2(\mathfrak{g}, \mathfrak{a}) 
\end{equation}
where the ring structures are those coming from Theorem \ref{1sttheorem}.
\end{corollary}
\begin{proof}
Take the odd part $\mathfrak{g}_1$ of $\mathfrak{g}$ to be trivial in \eqref{Theorem1}.
\end{proof}
\begin{corollary}
\label{cor lie algebra automor}
For an abelian extension of Lie algebra $  0\rightarrow \mathfrak{a} \xrightarrow{} \mathfrak{e} \xrightarrow{} \mathfrak{g} \to 0$ we have the following exact sequence of groups:
\begin{equation}
     1\rightarrow Aut^{\mathfrak{a},\mathfrak{g}}(\mathfrak{e})\xrightarrow{i} Aut^{\mathfrak{g}}_{\mathfrak{a}}(\mathfrak{e}) \xrightarrow{\widetilde{res}} Aut_\mathfrak{g}(\mathfrak{a})\xrightarrow{d} H^2(\mathfrak{g},\mathfrak{a}).
\end{equation}
\end{corollary}
\begin{proof}
Again take the odd part $\mathfrak{g}_1$ of $\mathfrak{g}$ to be trivial in \eqref{cor1}.
\end{proof}

\begin{remark}
With a slight modification wherever necessary we get the first exact sequence of \cite[Theorem 2.8]{Jamali} from our Corollary \ref{cor lie algebra automor} above.
\end{remark}
\begin{corollary}
For a Lie superalgebra $\mathfrak{g}$ and a module $\mathfrak{a}$ over $\mathfrak{g}$ if $H^2(\mathfrak{g}, \mathfrak{a})_0=0$ then any map  $\phi \in End_\mathfrak{g}(\mathfrak{a})$ can be extended to an endomorphism of $\mathfrak{e}$ inducing identity map on $\mathfrak{g}$.
\end{corollary}
\noindent Here we note that for any map $\phi \in End(\mathfrak{a})$ to get extended to a map in $End_\mathfrak{a}^\mathfrak{g}(\mathfrak{e})$ it is necessary that $\phi \in End_\mathfrak{g}(\mathfrak{a})$. Also, some of the Lie superalgebras and modules over them can be found in \cite{HH} for which $H^2(\mathfrak{g}, \mathfrak{a})_0=0$.
\begin{proof}
Since in this case the  map $d \equiv 0$ the lemma follows from exactness of \eqref{Theorem1}.
\end{proof}
\begin{corollary}
If in particular $  0\to \mathfrak{a} \to \mathfrak{e} \to \mathfrak{g} \to 0$ is a split exact sequence of Lie superalgebras then any map  $\phi \in End_\mathfrak{g}(\mathfrak{a})$ can be extended to a map in $End_\mathfrak{a}^\mathfrak{g}(\mathfrak{e})$.
\end{corollary}
\begin{proof}
If the exact sequence splits then we can choose a section $\eta:\mathfrak{g} \to \mathfrak{e}$ which is also a Lie superalgebra map, this implies that the map $\beta \equiv 0$. Then it follows from the description of the map $d$ in Lemma \ref{description of d} that $d \equiv 0$. Now the lemma follows from exactness of \eqref{Theorem1}. 
\end{proof}
\end{subsection}
\section{Construction of the second exact sequence}
\label{section 4}
In this section we establish our second exact sequence which is also given in \eqref{seq4}. This exact sequence helps us obtain some necessary and sufficient conditions for certain endomorphisms of $\mathfrak{g}$ to get lifted to that of $\mathfrak{e}$ fixing $\mathfrak{a}$ pointwise. The approach taken here is rather direct, considering appropriate maps between the monoids of endomorphisms and checking for exactness at each stage. We state our result in the next theorem:
\begin{theorem}
\label{Theorem2}
Let $0 \xrightarrow{i} \mathfrak{a} \to \mathfrak{e}\xrightarrow{p} \mathfrak{g} \to 0$ be an abelian extension of the Lie superalgebra $\mathfrak{g}$.
Then we have the following exact sequence of monoids:
\begin{equation}
\label{theorem2 seq}
     1\rightarrow End^{\mathfrak{a},\mathfrak{g}}(\mathfrak{e})\xrightarrow{i} End^{\mathfrak{a}}(\mathfrak{e}) \xrightarrow{\sigma} End^\mathfrak{a}(\mathfrak{g})\xrightarrow{\chi}H^2(\mathfrak{g},\mathfrak{a})_0 
\end{equation}
\noindent Here kernel of a monoid homomorphism $t:M\to N$ is defined to be the set $ker\,t:=\{m\in M \,|\, t(m)=e_N\}$ where $e_N$ is the identity of $N$.
\end{theorem}
Before we proceed to prove the above theorem, note that the first three terms are monoids with respect to the composition of maps with identity maps being the identity elements. The map $\sigma$ is the obvious map defined by $\sigma(\gamma):=\psi$ such that the following diagram is commutative:
\begin{center}
\begin{tikzcd}
   0\arrow{r} & \mathfrak{a}\arrow{r}{i}\arrow[equal]{d} & \mathfrak{e}\arrow{r}{p}\arrow{d}{\gamma} & \mathfrak{g}\arrow{r}\arrow[dotted]{d}{{\psi}} & 0 \\
  0\arrow{r} & \mathfrak{a}\arrow{r}{i} & \mathfrak{e}\arrow{r}{p} & \mathfrak{g}\arrow{r} & 0 
\end{tikzcd}
\end{center}
\noindent We have seen such a map $\psi$ can be described as $\psi= p \gamma s$ where $s$ is a section of $p$.

Now we prove a couple of simple lemmas which we shall use in the proof of Theorem \ref{Theorem2}.
\begin{lemma}
\label{lambda}
For every pair of endomorphisms $(\gamma,\psi)$ with $\sigma(\gamma)=\psi$ there is a homogeneous linear map $\lambda:\mathfrak{g}\to \mathfrak{a}$ of degree $0$ such that $\gamma(s(g))= \lambda(g)+ s\psi(g)$.
\end{lemma}
\noindent Such a map $\lambda$ has also been constructed in \cite[Theorem 1]{HH} and the lemma follows from there. However for the sake of completeness we prove it here.
\begin{proof}
Since $ps=1$ we can write $p \gamma s(g)=\psi(g)=ps \psi(g)$. This implies $p(\gamma s(g)-s\psi(g))=0$. So $\gamma s(g)-s\psi(g) \in \mathfrak{a}$. Now we define $\lambda(g):=\gamma s(g)-s\psi(g)\, \forall g\in \mathfrak{g}$. It can be very easily checked that $\lambda$ is homogeneous of degree $0$, hence the lemma follows.
\end{proof}

\begin{lemma}
\label{monoid homomorphism}
The image of the map $\sigma$ lies in $End^{\mathfrak{a}}(\mathfrak{g})$ and $\sigma:End^{\mathfrak{a}}(\mathfrak{e}) \to End^{\mathfrak{a}}(\mathfrak{g}) $ is a monoid homomorphism.
\end{lemma}

\begin{proof}
Let $\gamma \in End^{\mathfrak{a}}(\mathfrak{e})$, then $\gamma(a)=a\, \forall a\in \mathfrak{a}$. So
\begin{align*}
    [s(g), a]&= \gamma([s(g),a]) \\
             &=[\gamma(s(g)), a] \\
             &=[\lambda(g)+ s\psi(g) , a]\hspace{1cm} \text{(using Lemma \ref{lambda} )} \\
             &=[s\psi(g), a]  \hspace{.5cm}\forall\,g\in\mathfrak{g},\,a \in \mathfrak{a}.
\end{align*}
This implies $\psi(=\sigma(\gamma)) \in End^{\mathfrak{a}}(\mathfrak{g})$.

For the second part let $\gamma_1, \gamma_2 \in End^{\mathfrak{a}}(\mathfrak{e})$. Also let $\sigma(\gamma_i)=\psi_i$ for $i=1,2$ and the corresponding maps as in Lemma \ref{lambda} are $\lambda_1$ and $\lambda_2$. Then 
\begin{align*}
    \sigma(\gamma_1 \circ \gamma_2)(g)&= p (\gamma_1 \circ \gamma_2) s(g)\\
                                      &= p\gamma_1(\lambda_2(g)+ s\psi_2(g)) \hspace{1cm} \text{(using Lemma \ref{lambda} )}\\
                                      &=p( \lambda_2(g)+\lambda_1(\psi_2(g))+ s \psi_1(\psi_2(g))) \\
                                      &= p s \psi_1(\psi_2(g))\\
                                      &=\psi_1\circ\psi_2(g)=\sigma(\gamma_1) \circ \sigma(\gamma_2)(g).
\end{align*}
Also as $\sigma(id)=id$, the lemma follows.
\end{proof}
Now to each map in $End^{\mathfrak{a}}(\mathfrak{g})$ we associate a 2-cocycle in $H^2(\mathfrak{g}, \mathfrak{a})_0$ in the following way. Let $\psi \in End^{\mathfrak{a}}(\mathfrak{g})$. Define $\Psi: \mathfrak{g} \times \mathfrak{g}\to \mathfrak{a}$ by $\Psi(g,h):= \beta(\psi(g),\psi(h))-\beta(g,h)$;$g,h\in \mathfrak{g}$ where $\beta$ is defined in Lemma \ref{description of d}. It is well-known that $\beta$ is a homogeneous 2-cocycle of degree $0$ and with the help of that it can be easily proved that $\Psi$ is also a homogeneous 2-cocycle of degree $0$. So $[\Psi] \in H^2(\mathfrak{g}, \mathfrak{a})_0$ where $[\,\cdot\,]$ denotes the cohomology classes. Now we define $\chi:End^{\mathfrak{a}}(\mathfrak{g})\to  H^2(\mathfrak{g}, \mathfrak{a})_0$ given by $\chi(\psi):=\Psi$. Then as in \cite[Lemma 4.3]{HH} one can show that $\chi$ is well-defined that is it does not depend on the section used in the definition of $\beta$(denoted by $\theta$ in \cite{HH}). The map $\chi$ is not in general a  monoid homomorphism but its kernel will have usual meaning. Keeping the above discussions in mind now we proceed to prove Theorem \ref{Theorem2}:

\begin{proof}{(Proof of Theorem \ref{Theorem2})}

\noindent It is easy to see that $ker(\sigma)=End^{\mathfrak{a},\mathfrak{g}}(\mathfrak{e})$. So $i$ being the inclusion map clearly \eqref{theorem2 seq} is exact at the first and second term. Now to check exactness at the third term, let $End^{\mathfrak{a}}(\mathfrak{g}) \ni \psi= \sigma(\gamma)$ for some $\gamma \in End^{\mathfrak{a}}(\mathfrak{e})$. Then from Lemma \ref{lambda} we have $\gamma s(g)=\lambda(g)+ s \psi(g)$ for $g\in \mathfrak{g}$. Using this fact along with the facts that $\mathfrak{a}$ is abelian and $\gamma|_\mathfrak{a}=id$, we can write
\begin{align*}
   \beta(\psi(g),\psi(h))&=[s\psi(g), s\psi(h)]-s[\psi(g),\psi(h)]\\
   &= [s\psi(g), s\psi(h)]-s\psi([g,h])\\
   &=[\gamma s(g)-\lambda(g), \gamma s(h)-\lambda(h)]-\gamma s([g,h])+\lambda([g,h])\\
   &=[\gamma s(g), \gamma s(h)]-[\gamma s(g), \lambda(h)]+(-1)^{|g||h|}[\gamma s(h), \lambda(g)]-\gamma s([g,h])+\lambda([g,h])\\
   &= \gamma([s(g),s(h)]-s[g,h])-\gamma( [s(g), \lambda(h)])+(-1)^{|g||h|}\gamma([s(h), \lambda(g)])+\lambda([g,h]) \\
   &=\beta(g,h)-d(\lambda)(g,h)
\end{align*}
 for homogeneous elements $g,h \in \mathfrak{g}$ and then extend linearly, here $d$ is the coboundary map on the 1-cochains.
 
 \noindent Therefore $\chi(\psi)=[\Psi]=  [\beta(\psi(g),\psi(h))-\beta(g,h)]=[0]$, consequently $\psi\in ker\,\chi$ and $img \, \sigma \subseteq  ker\,\chi$. Now to prove the other inclusion let $\psi\in End^{\mathfrak{a}}(\mathfrak{g})$ be such that $\psi \in ker\, \chi$. Then using the definition of $\chi$, we can have a linear homogeneous map $\lambda:\mathfrak{g}\to \mathfrak{a}$ of degree 0 such that $\beta(\psi(g),\psi(h))-\beta(g,h)= -d(\lambda)(g,h)=-[s(g), \lambda(h)]+(-1)^{|g||h|}[s(h),\lambda(g)]+\lambda([g,h])$ for homogeneous $g,h \in \mathfrak{g}$. Then it turns out that the map $\gamma:\mathfrak{e}\to \mathfrak{e}$ defined by $\gamma(a\oplus s(g)):= a+ \lambda(g)+s\psi(g);a\in \mathfrak{a}, g\in \mathfrak{g}$ is a Lie superalgebra homomorphism. Moreover $\gamma(a)=a$ for $a\in\mathfrak{a}$ and $\sigma(\gamma)= p\gamma s=\psi$. Taking all these together we get $\gamma \in End^{\mathfrak{a}}(\mathfrak{g})$ and $ ker\,\chi\subseteq img \, \sigma$. Consequently \eqref{theorem2 seq} is exact at the third term, then result of Lemma \ref{monoid homomorphism} completes the proof of the theorem.
\end{proof}
Clearly the proof of the theorem above, unlike the proof of Theorem \ref{1sttheorem}, does not use any cohomological exact sequence and in view of this fact we ask the following question:
\begin{question}
Do we have any known cohomological exact sequence from which \eqref{theorem2 seq} would follow ?
\end{question}
\begin{remark}
The answer to this question for the case of groups is yes. An exact sequence like \eqref{theorem2 seq} for groups was obtained in \cite{Mariam} by making use of a non-abelian cohomology exact sequence though the article does not provide any description of the final map. Also that proof heavily relies on the availability of a set bijection $End(G)\simeq Z^1(G,G)$, the set of derivations or the (non-abelian) $1$-cocycles of $G$ with coefficients in $G$. However, we do not know of any such bijection for Lie superalgebras.
\end{remark}

Now we have the following immediate corollaries of the above theorem.
\begin{corollary}
Let  $0 \xrightarrow{i} \mathfrak{a} \to \mathfrak{e}\xrightarrow{p} \mathfrak{g} \to 0$ be an abelian extension of the Lie superalgebra $\mathfrak{g}$. Then the following is an exact sequence of groups:
\begin{equation}
\label{1st cor of theorem2}
     1\rightarrow Aut^{\mathfrak{a},\mathfrak{g}}(\mathfrak{e})\xrightarrow{i} Aut^{\mathfrak{a}}(\mathfrak{e}) \xrightarrow{\sigma} Aut^\mathfrak{a}(\mathfrak{g})\xrightarrow{\chi}H^2(\mathfrak{g},\mathfrak{a})_0 
\end{equation}
\end{corollary}
\begin{proof}
In Theorem \ref{Theorem2} above if we restrict our maps to the corresponding groups of invertible endomorphisms of the monoids we get the required exact sequence.
\end{proof}
\begin{remark}
The exact sequence \eqref{1st cor of theorem2} resembles the second exact sequence obtained in \cite[Theorem 4.6]{HH}. In that sense Theorem \ref{Theorem2} generalises the second exact sequence of \cite[Theorem 4.6]{HH}. The only point we should note here is that the maps $\chi$ in Theorem \ref{Theorem2} (and hence in \eqref{1st cor of theorem2}) and $\lambda_2$ in \cite[Theorem 4.6]{HH} are not exactly the same. Nevertheless, the above corollary also gives necessary and sufficient conditions for the lifting of automorphisms of $\mathfrak{g}$.
\end{remark}
\begin{corollary}
Let  $0 \xrightarrow{i} \mathfrak{a} \to \mathfrak{e}\xrightarrow{p} \mathfrak{g} \to 0$ be an abelian extension of the Lie superalgebra $\mathfrak{g}$. If this extension splits or if $H^2(\mathfrak{g},\mathfrak{a})=0$ then any endomorphism of $\mathfrak{g}$ which keeps the action of $\mathfrak{g}$ on $\mathfrak{a}$ invariant can be lifted to $\mathfrak{g}$ fixing $\mathfrak{a}$ pointwise. Moreover if the extension is central then any endomorphism of $\mathfrak{g}$ can be lifted to $\mathfrak{g}$ fixing $\mathfrak{a}$ pointwise.
\end{corollary}
\begin{proof}
Clearly if the extension splits then $\beta\equiv 0$ which in turn implies $\chi \equiv 0$, which is also the case when $H^2(\mathfrak{g},\mathfrak{a})=0$ itself. This implies the image of the map $\sigma$ is whole of $End^\mathfrak{a}(\mathfrak{g})$ and the first part of the lemma is proved. Now in addition if the extension is central then the action of $\mathfrak{g}$ on $\mathfrak{a}$ becomes trivial and in that case $End^\mathfrak{a}(\mathfrak{g})=End(\mathfrak{g})$ from which the second part of the lemma also follows.
\end{proof}

\begin{section}{Application to semidirect products}
\label{section 5}
In this section as an application of our results we describe certain automorphism groups of semidirect products of Lie superalgebras. Let $\mathfrak{g}$ be a Lie superalgebra and $\mathfrak{a}$ a module over $\mathfrak{g}$. Then the semidirect product of $\mathfrak{a}$ and $\mathfrak{g}$ is another Lie superalgebra denoted by $ \mathfrak{g} \ltimes\mathfrak{a}$ which is the vector space $ \mathfrak{g}  \oplus \mathfrak{a}$ on which the bracket is defined by $[(x,a),\,(y,b)]:=( [x,y],\,x\cdot b-(-1)^{|a||y|}y\cdot a )$ for all homogeneous elements $x,y \in \mathfrak{g};a,b \in \mathfrak{a}$ and then extended linearly. Then $\mathfrak{a}$ can be seen as an abelian ideal in $ \mathfrak{g} \ltimes\mathfrak{a}$ while $\mathfrak{g}$ is just a subalgebra. Corresponding to this semidirect product one obtains the following abelian extension of the Lie superalgebra $\mathfrak{g}$:
 \begin{equation}
 \label{semidirect product exact sequence}
     0 \to \mathfrak{a} \xrightarrow{a\mapsto(0,a)} \mathfrak{g} \ltimes\mathfrak{a} \xrightarrow{(g,a)\mapsto g} \mathfrak{g} \to 0
 \end{equation}
 Then it happens that the induced action of $\mathfrak{g}$ on $\mathfrak{a}$ coming from \eqref{semidirect product exact sequence} coincides with the action we started with. Also \eqref{semidirect product exact sequence} splits, a section which is also a Lie superalgebra homomorphism is given by the map $s:\mathfrak{g} \to  \mathfrak{g} \ltimes\mathfrak{a}$ defined by $s(g):=(g,0)$. Now we have the following description of certain automorphism groups:
 \begin{theorem}
 \label{theorem3}
Let $\mathfrak{g}$ be a Lie superalgebra and $\mathfrak{a}$ a module over $\mathfrak{g}$. Then we have the following isomorphisms of groups:
 \begin{enumerate}
      \item $Aut^{\mathfrak{g}}_{\mathfrak{a}}(\mathfrak{g} \ltimes\mathfrak{a})\simeq  Aut_\mathfrak{g}(\mathfrak{a}) \ltimes Aut^{\mathfrak{a},\mathfrak{g}}(\mathfrak{g} \ltimes\mathfrak{a})$  
     \item  $Aut^{\mathfrak{a}}(\mathfrak{g} \ltimes\mathfrak{a}) \simeq    Aut^\mathfrak{a}(\mathfrak{g}) \ltimes Aut^{\mathfrak{a},\mathfrak{g}}(\mathfrak{g} \ltimes\mathfrak{a})$
 \end{enumerate}
 \end{theorem}
\noindent The theorem says that certain automorphism groups of semidirect products are again semidirect products.

\begin{proof}
From the above discussion, corresponding to a semidirect product $\mathfrak{g} \ltimes\mathfrak{a}$ we have the abelian extension given in \eqref{semidirect product exact sequence}.

Since \eqref{semidirect product exact sequence} splits, the corresponding map $d:Aut_\mathfrak{g}(\mathfrak{a})\xrightarrow{} H^2(\mathfrak{g},\mathfrak{a})_0$ in \eqref{cor1} becomes trivial and \eqref{cor1} takes the following form in this case:
\begin{equation}
\label{1st split exact}
    1\rightarrow Aut^{\mathfrak{a},\mathfrak{g}}(\mathfrak{g} \ltimes\mathfrak{a})\xrightarrow{i} Aut^{\mathfrak{g}}_{\mathfrak{a}}(\mathfrak{g} \ltimes\mathfrak{a}) \xrightarrow{\widetilde{res}} Aut_\mathfrak{g}(\mathfrak{a})\to 0
\end{equation}
\noindent Now we prove that the above exact sequence also splits. For that let us consider the map $\epsilon :Aut_\mathfrak{g}(\mathfrak{a})\to Aut^{\mathfrak{g}}_{\mathfrak{a}}(\mathfrak{g} \ltimes\mathfrak{a})$ given by $\epsilon(\phi):=\gamma$ for $\phi \in Aut_\mathfrak{g}(\mathfrak{a}) $ where $\gamma$ is defined by $\gamma(g,a):=(g, \phi(a))$ for all $g\in \mathfrak{g},\, a \in \mathfrak{a}$. We need to check that $\gamma$ is indeed an element of $Aut^{\mathfrak{g}}_{\mathfrak{a}}(\mathfrak{g} \ltimes\mathfrak{a})$. Clearly $\gamma(\mathfrak{a})=\mathfrak{a}$ and $\gamma$ induces identity map on $\mathfrak{g}$. To prove that $\gamma$ is a homomorphism, let $x,y \in \mathfrak{g};a,b \in \mathfrak{a}$ be homogeneous elements. Then 
\begin{align*}
    \gamma([(x,a),(y,b)]) &= \gamma([x,y],\, x\cdot b - (-1)^{|y||a|} y \cdot a)\\
                          &= ([x,y],\, \phi(x\cdot b - (-1)^{|y||a|} y \cdot a))\\
                          &=([x,y],\,x\cdot\phi(b)- (-1)^{|y||a|}y \cdot \phi(a)) \text{      [as $\phi\in Aut_\mathfrak{g}(\mathfrak{a})$]}\\
                          &=[(x,\phi(a)),\,(y,\phi(b))]\\
                          &=[\gamma(x,a),\,\gamma(y, b)]
\end{align*}
So $\gamma$ is a Lie superalgebra homomorphism. Also since $\gamma$ is a bijection, $\phi$ being a bijection, $\gamma \in Aut^{\mathfrak{g}}_{\mathfrak{a}}(\mathfrak{g} \ltimes\mathfrak{a})$.
Moreover, one can easily check using \eqref{ker in cor1} that $\widetilde{res}\circ \epsilon$ is the identity map of $\mathfrak{g}$, so $\epsilon$ is a section. It is also easy to check that $\epsilon$ is a group homomorphism. Consequently the exact sequence \eqref{1st split exact} splits and the first isomorphism of the theorem follows.

Since \eqref{semidirect product exact sequence} splits we also have that the corresponding map $\chi: Aut^{\mathfrak{a}}(\mathfrak{g}) \to H^2(\mathfrak{g},\mathfrak{a})_0$ in \eqref{1st cor of theorem2} is trivial and the sequence \eqref{1st cor of theorem2} becomes
\begin{equation}
\label{2nd split exact}
    1\rightarrow Aut^{\mathfrak{a},\mathfrak{g}}(\mathfrak{g} \ltimes\mathfrak{a})\xrightarrow{i} Aut^{\mathfrak{a}}(\mathfrak{g} \ltimes\mathfrak{a}) \xrightarrow{\sigma} Aut^\mathfrak{a}(\mathfrak{g})\to 0.
\end{equation}
Again, to prove that \eqref{2nd split exact} splits we consider a map $\alpha:Aut^\mathfrak{a}(\mathfrak{g}) \to Aut^{\mathfrak{a}}(\mathfrak{g} \ltimes\mathfrak{a})$ defined by $\alpha(\psi):=\gamma$, here $\gamma$ is defined by $\gamma(g,a):=(\psi(g),a)$. Clearly $\gamma$ fixes $\mathfrak{a}$ pointwise and $\gamma$ is a bijection. To prove $\gamma$ is a homomorphism, let us again take  homogeneous elements $x,y \in \mathfrak{g};a,b \in \mathfrak{a}$. Then 
\begin{align*}
    \gamma([(x,a),(y,b)]) &= \gamma([x,y],\, x\cdot b - (-1)^{|y||a|} y \cdot a)\\
                          &= (\psi([x,y]),\, x\cdot b - (-1)^{|y||a|} y \cdot a)\\
                          &=([\psi(x), \psi(y)],\,\psi(x)\cdot b - (-1)^{|\psi(y)||a|}\psi(y) \cdot a) \text{      [as $\psi\in Aut^\mathfrak{a}(\mathfrak{g})$]   }\\
                          &=[(\psi(x),a),\,(\psi(y),b)]\\
                          &=[\gamma(x,a),\,\gamma(y, b)]
\end{align*}
So $\gamma$ is a homomorphism. In a similar way to the previous part one can also check that $\alpha$ is a section as well as a homomorphism. Therefore the sequence \eqref{2nd split exact} also splits and we have the second isomorphism of groups.
This completes the proof of the theorem.
\end{proof}
 The exact sequence \eqref{semidirect product exact sequence} which was playing the role of abelian extension in the above theorem was of very particular type. For a general abelian extension $0\to\mathfrak{a}\to\mathfrak{e}\to\mathfrak{g}\to 0$, the group isomorphisms of the above theorem might not hold true. For example, the first example below shows that first isomorphism of Theorem \ref{theorem3} does not hold true in general.
 
 In both the examples below the underlying field $F$ is of characteristic $0$, as was assumed earlier.
 \begin{example} 
 Consider the $3$-dimensional Heisenberg Lie algebra $\mathfrak{h}_3$ which is $span\{x,y,z\}$ as vector space in which the only non-zero bracket is $[x,y]=z$. Clearly the center $Z(\mathfrak{h}_3)=span\{z\}$. Also let $Ab(n)$ denotes the $n$-dimensional abelian Lie algebra. Then $\mathfrak{h}_3$ can be seen as $1$-dimensional central extension of the $2$-dimensional abelian Lie algebra in the following way where $q$ is the obvious quotient map:
 \begin{equation}
 \label{Heisenberg extension}
     0\to Z(\mathfrak{h}_3)\xrightarrow{i} \mathfrak{h}_3 \xrightarrow{q} Ab(2)\to 0
 \end{equation}
 In this case  $Aut^{Z(\mathfrak{h}_3),Ab(2)}(\mathfrak{h}_3)= Aut_{Z(\mathfrak{h}_3)}^{Ab(2)}(\mathfrak{h}_3)=\left\{\begin{pmatrix}
  1 & 0 & 0\\ 
  0 & 1 & 0 \\
  a & b & 1
\end{pmatrix} : a,b \in F \right \}\simeq F\oplus F$ and $Aut_{Ab(2)}(Z(\mathfrak{h}_3))=F^\ast$, the group of non-zero elements of $F$.
Now since $Aut_{Z(\mathfrak{h}_3)}^{Ab(2)}(\mathfrak{h}_3)$ is abelian, if $Aut_{Z(\mathfrak{h}_3)}^{Ab(2)}(\mathfrak{h}_3)$ is a semidirect product then it has to be the direct product in particular. But that would imply $ F\oplus F \simeq F^\ast \oplus F\oplus F $ which can not be true as the later has an element of order 2. So $Aut_{Z(\mathfrak{h}_3)}^{Ab(2)}(\mathfrak{h}_3)\not\simeq Aut_{Ab(2)}(Z(\mathfrak{h}_3)) \ltimes Aut^{{Z(\mathfrak{h}_3)},Ab(2)}(\mathfrak{h}_3)$.
\end{example}
 In the example above one checks that \eqref{Heisenberg extension} does not split and hence $\mathfrak{h}_3$ is not a semidirect product. However, it is not necessary for $\mathfrak{e}$ to be a semidirect product for the isomorphism $Aut^{\mathfrak{a}}(\mathfrak{e}) \simeq Aut^\mathfrak{a}(\mathfrak{g}) \ltimes Aut^{\mathfrak{a},\mathfrak{g}}(\mathfrak{e})$ to hold true. The next example shows that although $Aut^{\mathfrak{a}}(\mathfrak{e}) \simeq Aut^\mathfrak{a}(\mathfrak{g}) \ltimes Aut^{\mathfrak{a},\mathfrak{g}}(\mathfrak{e})$, $\mathfrak{e}\not\simeq \mathfrak{g}\ltimes\mathfrak{a}$.
 \begin{example}
 Let us consider the Lie superalgebra $\mathfrak{ba}_1$ having basis $\{x\,|\,y,z\}$ and the only non-zero bracket is given by $[x,y]=z$. This is the Lie superalgebra corresponding to $n=1$ from the family $\mathfrak{ba}_n$ given in \cite[Section 6]{HH}. Here the bar in the basis separates even elements from odd, $|x|=0$ and $|y|=|z|=1$. Also the center $Z(\mathfrak{ba}_1)=span\{z\}$. Now for simplicity rename $\mathfrak{ba}_1$ to be $\mathcal{H}$, then $\mathcal{H}$ can be seen as a central extension of an abelian Lie superalgebra in the following way:
 \begin{equation}
     0\to Z(\mathcal{H})\xrightarrow{i} \mathcal{H} \xrightarrow{q}  Ab(1|1)\to 0
 \end{equation}
 where $q$ is the obvious quotient map and $Ab(1|1)$ is the abelian Lie superalgebra of superdimension $(1,1)$. Then one has $Aut^{Z(\mathcal{H}), Ab(1|1)}(\mathcal{H})=\left\{\begin{pmatrix}
 \begin{array}{c|c c}
    1  &   &  \\
    \hline
       & 1 & 0\\
       & a & 1
 \end{array}
 \end{pmatrix}: a\in F \right\}$, $Aut^{Z(\mathcal{H})}(\mathcal{H})=\left\{\begin{pmatrix}
 \begin{array}{c|c c}
    b  &   &  \\
    \hline
       & c & 0\\
       & d & 1
 \end{array}
 \end{pmatrix}: b,c,d\in F; bc\neq 0 \right\}$ and \\
 $Aut^{Z(\mathcal{H})}({Ab(1|1)})=\left\{\begin{pmatrix}
 \begin{array}{c|c}
    e  &   \\
    \hline
       & f \\
      
 \end{array}
 \end{pmatrix}: e,f \in F;ef\neq 0 \right\}$.

It can now easily be checked that the following is an exact sequence of groups:
\begin{equation}
\label{Heisenberg split 2}
    1\to Aut^{Z(\mathcal{H}), Ab(1|1)}(\mathcal{H}) \xrightarrow{i} Aut^{Z(\mathcal{H})}(\mathcal{H}) \xrightarrow{\sigma} Aut^{Z(\mathcal{H})}({Ab(1|1)}) \to 1
\end{equation}
where $i$ is the obvious inclusion map and the map $\sigma$ can be given by 
$\sigma(\begin{pmatrix}
 \begin{array}{c|c c}
    b  &   &  \\
    \hline
       & c & 0\\
       & d & 1
 \end{array}
 \end{pmatrix})= \begin{pmatrix}
 \begin{array}{c|c}
    b  &   \\
    \hline
       & c \\
       \end{array}
 \end{pmatrix}$.
 Now consider the map $\epsilon:Aut^{Z(\mathcal{H})}({Ab(1|1)}) \to Aut^{Z(\mathcal{H})}(\mathcal{H})$ given by $\epsilon(\begin{pmatrix}
 \begin{array}{c|c}
    e  &   \\
    \hline
       & f \\
       \end{array}
 \end{pmatrix}$)= $\begin{pmatrix}
 \begin{array}{c|c c}
    e  &   &  \\
    \hline
       & f & 0\\
       & 0 & 1
 \end{array}
 \end{pmatrix}$. Then $\epsilon$ is a section of the map $\sigma$, also $\epsilon$ is a group homomorphism. Therefore \eqref{Heisenberg split 2} splits and consequently $ Aut^{Z(\mathcal{H})}(\mathcal{H}) \simeq Aut^{Z(\mathcal{H})}({Ab(1|1)}) \ltimes Aut^{Z(\mathcal{H}), Ab(1|1)}(\mathcal{H})$. However, $\mathcal{H}\not\simeq Ab(1|1) \ltimes Z(\mathcal{H})$ as that would imply $\mathcal{H}$ has a $2$-dimensional abelian subalgebra not containing the center.
 \end{example}

\end{section}

\section*{Acknowledgements}
Both the authors would like to thank Prof. Ian M. Musson for fruitful communications.

\section*{Disclosure statement}
There is no conflict of interest.

\section*{Funding}
The first author acknowledges Ph.D fellowship from NBHM, DAE, Govt of India and partial financial support from Shiv Nadar University.


\begin{thebibliography}{99}


\bibitem{HH}
 Hazra~SK, Habib~A. Wells exact sequence and automorphisms of
extensions of Lie superalgebras. J. Lie Theory. 2020;30(1):179--199.

\bibitem{Wells}
Wells~C. Automorphisms of group extensions. Trans. Amer. Math. Soc. 1971;155:189--194.

\bibitem{Passi}
Passi~IBS, Singh~M, Yadav~MK. Automorphisms of abelian group extensions. J. Algebra.  2010;324(4):820--830.

\bibitem{Jill}
Dietz~J. Resurrecting Wells' exact sequence and Buckley's group action. Groups St Andrews 2013, 209–224, London Math. Soc. Lecture Note Ser., 422, Cambridge Univ. Press, Cambridge, 2015.

\bibitem{Jamali}
Fouladi~S, Jamali~AR, Orfi~R. Automorphisms of abelian Lie ring extensions. J. Algebra Appl. 2017;16(9)1750176:12 pp.

\bibitem{Jamali2}
Jamali~AR. The Wells exact sequence for the automorphism group of a Lie ring extension. J. Algebra Appl. 2019;18(3)1950058:15 pp.

\bibitem{Singh}
Singh~M, Bardakov~VG. Extensions and automorphisms of Lie algebras. J. Algebra Appl. 2017;16(9)1750162: 15 pp.

\bibitem{Xu}
Xu~S, Tan~Y. The Wells map for abelian extensions of 3-Lie algebras. Czechoslovak Math. J. 2019;69(144)(4):1133--1164.

\bibitem{Du}
Du~L, Tan~Y. Wells sequences for abelian extensions of Lie coalgebras. J. Algebra Appl. 2021;2150149, 41 pp.  DOI: 10.1142/S0219498821501498

\bibitem{Barati}
Barati~M, Saeedi~F. Derivations of abelian Lie algebra extensions. (English summary) Note Mat. 2019;39(2):71--86.

\bibitem{Mariam}
Pirashvili~M. Endomorphisms in short exact sequences. Comm. Algebra. 2017;45(1):312--321. 


\bibitem{Musson}
Musson~IM. Lie superalgebras and enveloping algebras.  Graduate Studies in Mathematics, 131. American Mathematical Society, Providence, RI; 2012.

\bibitem{Kaplansky}
Kaplansky~I. Fields and Rings. The University of Chicago Press; 1969.

\end{thebibliography}
\end{document}